\documentclass[11pt,a4paper,twoside]{amsart}

\usepackage{graphicx}
\usepackage{amsfonts}
\usepackage{amssymb}
\usepackage{amsthm}
\usepackage[centertags]{amsmath}
\usepackage{newlfont}
\usepackage{hyperref}
\usepackage{color}
\usepackage{graphicx,color}
\usepackage{amsmath, amssymb, graphics}

\def\r{\mathbb R}
 
\def\s{\mathbb S}
\def\t{\mathbf t}
\def\n{\mathbf n}
\def\b{\mathbf b}
\def\c{\mathbf c}

\newtheorem{theorem}{Theorem}[section]

\newtheorem{corollary}[theorem]{Corollary}

\theoremstyle{definition}
\newtheorem{definition}[theorem]{Definition}

\newtheorem{remark}[theorem]{Remark}

\title {Circles-foliated stationary surfaces of the Dirichlet energy} 
\author{Rafael L\'opez}  
\address{Departamento de Geometr\'{\i}a y Topolog\'{\i}a\\  Universidad de Granada. 18071 Granada, Spain} 
\email{rcamino@ugr.es}



\begin{document} 
 
\begin{abstract}  
In Euclidean space we study surfaces with constant anisotropic mean curvature $\Lambda$ of the Dirichlet  energy $\int_\Omega( |Du|^2+\Lambda u)$.  We prove the existence of non-rotational surfaces with $\Lambda=0$ and foliated  by a one-parameter family of circles contained in horizontal planes obtaining a   geometric description of them. These surfaces extend the known Riemann  examples of the theory of minimal surfaces to the anisotropic context of the Dirichlet energy. More general, we   classify all  surfaces with zero anisotropic mean curvature   foliated by circles  proving that either the surface is axially symmetric about  the $z$-axis or the surface belongs to one of the above examples. We also study the case that the anisotropic mean curvature is a non-zero constant. 
\end{abstract} 
\maketitle

\section{Introduction and statement of the results} \label{intro}

Let $\Sigma$ be a surface in Euclidean space $\r^3$ given as the graph of a smooth function $u=u(x,y)$ defined in a bounded domain $\Omega\subset\r^2$, where $(x,y,z)$ stand for the canonical coordinates of $\r^3$. If $\Lambda\in\r$, consider the energy functional 
$$E[u]=\int_\Omega |Du|^2+\Lambda\int_\Omega u.$$
The first term of $E$ is the {\it Dirichlet energy} of $u$. The integral $\int_\Omega u$ represents the volume enclosed by $\Sigma$ with $\Omega\times\{0\}$. Viewing $\Lambda$ as a Lagrange multiplier,  a critical point $\Sigma$ of $E$ is a critical point of the Dirichlet energy for variations that preserve the volume of $\Sigma$. Using calculus of variations, $u=u(x,y)$ is a critical point of $E$ 
if and only if $u_{xx}+u_{yy}=\Lambda/2$. Thus if $\Lambda=0$, critical points   are locally graphs of harmonic functions on domains of $\r^2$.  We can write $E[u]$ as an integral on $\Sigma$ by observing that the unit normal $\nu=(\nu_1,\nu_2,\nu_3)$ of $\Sigma$ is $\nu=(-Du,1)/\sqrt{1+|Du|^2}$. Then the Dirichlet energy is a functional depending on $\nu_3$, namely,  
\begin{equation}\label{fu}
\mathcal{F}(\Sigma)=\int_\Sigma \left(\dfrac{1}{\nu_3}-\nu_3\right)\, d\Sigma,
\end{equation}
 where $d\Sigma$ is the area element of $\Sigma$. More generally, if $X\colon\Sigma\to\r^3$  is an immersion of an oriented surface,   one can consider energy functionals of type  $\mathcal{F}(X)=\int_\Sigma F(\nu)\, d\Sigma$, where $F\colon U\subset\s^2\to\r^+$ is a positive smooth function on the unit sphere $\s^2$. These functionals appear in some fluid phenomenons in physics and are called of {\it anisotropic} type because the surface tensions of interfaces $\Sigma$ depend on the normal direction $\nu$ of $\Sigma$ \cite{ta}. 
 
The function $F$ given in \eqref{fu}  satisfies an elliptic condition in the sense that the matrix  $A:=D^2F+F\, \mbox{Id}$ is positive definite, where  $D^2F$ is the Hessian of $F$. Ellipticity of $F$ implies that the map $\xi\colon\s^2\to\r^3$ defined by $\xi(\nu)=DF(\nu)+F(\nu)\nu$, where $DF$ is the gradient of $F$ in $\s^2$, defines a   convex surface called the {\it Wulff shape}. For general    elliptic energies $\mathcal{F}$, if  the domain of $F$ is $U=\s^2$, then the Wulff shape is an ovaloid. The Wulff shape is the unique global minimizer to $\mathcal{F}$ when a volume constraint is fixed. For the Dirichlet energy \eqref{fu}, and since $\nu_3\not=0$, the function $F$ is   defined only in a hemisphere of $\s^2$ being the Wulff shape  the paraboloid of equation $z=x^2+y^2$, after translations and rescalings  in $\r^3$. It was in \cite{re} where critical points of energies $\mathcal{F}(X)=\int_\Sigma F(\nu)\, d\Sigma$ were studied from the differential-geometric view point. The Dirichlet energy  \eqref{fu} appears in   \cite[p. 709]{re}.  
 
Critical points of a functional $\mathcal{F}$ for compactly supported volume-preserving variations are characterized by the property that the function 
$$\mbox{trace}_\Sigma Ad\nu=-2HF+\mbox{div}_\Sigma DF$$
 is constant, where $H$ is the mean curvature of $X$. 
 
 \begin{definition} Let $\mathcal{F}$ be  an anisotropic surface energy. The {\it anisotropic mean curvature} of a surface $\Sigma$  immersed in $\r^3$ is defined by 
\begin{equation}\label{L}
\Lambda=2HF-\mbox{div}_\Sigma DF.
\end{equation}
\end{definition}
A {\it stationary surface}, or critical point, of the functional $\mathcal{F}$ is a surface with constant anisotropic mean curvature (CAMC). This  extends the notion of the surfaces of constant mean curvature (CMC) because when    $F\equiv 1$, the functional $\mathcal{F}$ is simply the area of the surface and  $\Lambda=2H$.   A CAMC surface with $\Lambda=0$ is called   {\it anisotropic minimal surface}.   By its implications in physics, the stability of CAMC surfaces  has received special attention. We refer to the reader a series of works by  Koiso and Palmer   \cite{kp2,kp3,kp5} and references therein.  See also \cite{bs,gmt,gx,jw,ro}.

In this paper we ask for the existence of stationary surfaces of the Dirichlet energy \eqref{fu}    constructed by a foliation of spatial circles of $\r^3$.   The motivation of this question comes from the theory of   CMC  surfaces. In general, a surface $\Sigma$ of $\r^3$ is said to be a {\it cyclic surface} if $\Sigma$ is  foliated by a smooth one-parameter family of circles, where the circles   need not have the same radius.   Each one of the planes containing the circles are called planes of the foliation. A cyclic surface can be also viewed as a surface constructed by the movement in $\r^3$ of a circle whose radius and center go changing along the movement. First examples of cyclic surfaces are the surfaces of revolution. In this case, the circles are   situated in parallel planes and all them are coaxial, that is, the curve of centers is a straight-line orthogonal to each plane of the foliation. 

The classification of cyclic surfaces with constant mean curvature $H$ is known. If the surface is rotational, then it is a Delaunay surface: planes and catenoid ($H=0$) and cylinders, spheres, unduloids and nodoids  ($H\not=0$).  Besides the surfaces of revolution, the only cyclic CMC  surfaces are a family of minimal surfaces ($H=0$) discovered by Riemann \cite{ri}  and the round sphere $(H\not=0)$: \cite{en1,en2,ni}. Notice that any smooth one-parameter family of planes intersecting a sphere makes a foliation of the sphere by circles. The Riemann minimal examples play a remarkable role in the theory of minimal surfaces. These surfaces are embedded with an infinite number of planar ends. They alsoare  periodic along a discrete group of translations. Generalizations of this classification to arbitrary dimensions and other ambients spaces have been  obtained. Without aiming to give a complete list, we refer  \cite{ja1,ja2,lls,lo2,lo3,lo1,pa}.

The objective in this paper is the classification of all   cyclic CAMC surfaces of the Dirichlet energy \eqref{fu}. Specially, our interest are the anisotropic minimal surfaces according to the Riemann minimal examples in the isotropic case.  

The main result in this paper is to prove the existence of a family of non-rotational anistropic minimal surface foliated by circles contained in parallel planes. In fact, these planes must be  parallel to the $xy$-plane. However, if the anisotropic mean curvature is a non-zero constant, we prove that the surface must be rotational.  

\begin{theorem}\label{T1}
Let $\Sigma$ be an CAMC   surface of the Dirichlet energy \eqref{fu}. Assume that  $\Sigma$ is foliated by circles contained in planes parallel to the $xy$-plane.
\begin{enumerate}
\item If $\Lambda\not=0$, then  $\Sigma$ is a surface of revolution about  an axis parallel to  the $z$-axis.
\item If $\Lambda=0$, then $\Sigma$ is a surface of revolution about an axis parallel to the $z$-axis (a horizontal plane $z=c$ or  $z=c\log(\sqrt{x^2+y^2})$, $c\not=0$) or the surface is parametrized by 
$$X(s,\theta)=(a(s),b(s),s)+r(s)(\cos\theta,\sin\theta,0),$$
where we have three types of families of surfaces. After translations of $\r^3$, they are the following:
\begin{enumerate}
\item Type I. \begin{equation}\label{sol1}
r(s) =\frac{ c}{\sqrt{\lambda^2+\mu^2}\cos (cs)},\quad 
 a(s)=\frac{ c\lambda \tan(cs)}{\lambda^2+\mu^2},  \quad b(s)=\frac{ c\mu \tan(cs)}{\lambda^2+\mu^2}.
  \end{equation}
 \item Type II.
\begin{equation}\label{sol2}
r(s) =\frac{ 1}{\sqrt{\lambda^2+\mu^2}s+c},\quad 
 a(s)=-\frac{ \lambda  }{(\lambda^2+\mu^2)s+c\sqrt{\lambda^2+\mu^2}}, \quad  b(s)=-\frac{ \mu  }{(\lambda^2+\mu^2)s+c\sqrt{\lambda^2+\mu^2}}.
  \end{equation}
 
\item Type III.

\begin{equation}\label{sol3}
r(s) =\frac{ c}{\sqrt{\lambda^2+\mu^2}\sinh (cs)},\quad 
 a(s)=-\frac{ c\lambda \coth(cs)}{\lambda^2+\mu^2}, \quad  b(s)=-\frac{ c\mu \coth(cs)}{\lambda^2+\mu^2}.
  \end{equation}

\end{enumerate}
For all these solutions,     $c,\lambda,\mu\in\r$ with the condition $c,\lambda^2+\mu^2\not=0$. The domain $I\subset\r$ of $s$  is given to make sense to the functions $r$, $a$ and $b$. In the three cases, all surfaces are embedded and the curve of centers is included in a vertical plane, which it is a plane of symmetry of the surface.

\end{enumerate}
\end{theorem}

Following with the particular case $\Lambda=0$, the following result proves that for a cyclic anisotropic minimal surface,  the planes of the foliation  are all parallel to the $xy$-plane.

\begin{theorem} \label{T2}
Let $\Sigma$ be an anisotropic minimal surface   of the Dirichlet energy \eqref{fu}. If $\Sigma$ is   a cyclic surface, then the planes of the foliation are all parallel to the $xy$-plane.  
\end{theorem}

As an immediate consequence of Thms. \ref{T1} and \ref{T2}, we have the complete classification of the  cyclic anisotropic minimal surfaces of the Dirichlet energy. 

\begin{corollary} \label{cor0}
If $\Sigma$ be a   cyclic anisotropic minimal   surface of the Dirichlet energy \eqref{fu}, then $\Sigma$ is a surface of revolution about an axis parallel to the $z$-axis or it is a member of the family of surfaces that appear in (2) of Thm. \ref{T1}.
\end{corollary}

When the anisotropic mean curvature $\Lambda$ is a non-zero constant, we also obtain results of classification. However, as we will see  along the development of the proofs, we will need   to calculate the anisotropic mean curvature $\Lambda$   for given parametrizations of the surfaces. This involves long and tedious computations. We have used  a  symbolic software as  Mathematica to handle these cumbersome expressions  \cite{wo}.  In particular, when $\Lambda\not=0$, all these computations are more difficult than in the case $\Lambda=0$. We summarize the results when $\Lambda\not=0$.

\begin{theorem}\label{T3}
Let $\Sigma$ be a  cyclic CAMC surface   of the Dirichlet energy \eqref{fu}. If $\Lambda\not=0$, then  the planes of the foliation must be parallel.
In addition, if these planes   are parallel to the $xy$-plane, then $\Sigma$ is  a surface of revolution.
\end{theorem}

Notice that the second part of this theorem coincides with the statement (1) of Thm. \ref{T1}. For a full classification when $\Lambda\not=0$, it would only remain to prove that when  the planes of the foliation are parallel, then they are parallel to the $xy$-plane.    

 We point out that CAMC surfaces of revolution   have been also studied under the natural hypothesis that the surface energy $\mathcal{F}$ is axially symmetric, that is, it is invariant by rotations about an axis $L$ of $\r^3$ \cite{kp1,kp2,kp4}.  If   $L$ is the $z$-axis, then we demand that  the anisotropic surface energy  is of type
\begin{equation}\label{fr}
\mathcal{F}(X)=\int_\Sigma F(\nu_3)\, d\Sigma.
\end{equation}
The Dirichlet energy  \eqref{fu} is thus axially symmetric with respect to the $z$-axis. However, notice that in principle there is not an {\it a priori} relation between  the axis of $\mathcal{F}$ and the rotation axis of a  CAMC surface of revolution. For the Dirichlet energy, we prove that both axes must be parallel.

\begin{theorem}\label{ROT1}
 Let $\Sigma$ be a CAMC surface of the   Dirichlet energy  \eqref{fu}. If $\Sigma$ is a surface of revolution, then its rotation axis is parallel to the $z$-axis.
\end{theorem}

The proof of   Thm.   \ref{T1}  will be done in Sect. \ref{s3}, after a Preliminary section \ref{s2}. In the final part of Sect. \ref{s3}, we describe  properties of the surfaces parametrized by \eqref{sol1}-\eqref{sol2}-\eqref{sol3}, some of them are shared by the Riemann minimal examples. Theorem \ref{T2} is proved in Sect. \ref{s4} and it consists of two steps. First, we will prove that the planes of the foliation must be parallel. In this step, we will also assume the case that $\Lambda$ is an arbitrary constant and not necessarily $0$. In particular, this proves the first statement of Thm. \ref{T3}. In the second step, we prove that these planes must be parallel to the $xy$-plane.  Theorem \ref{ROT1} is proved in Sect. \ref{s5} where previous computations of Sect. \ref{s3} will provide a short proof of the theorem.

It would be desirable to extend all these results to other   anisotropic surface energies.   For example, Reilly also considered in \cite{re} the energy 
\begin{equation*}
\mathcal{F}(\nu_3)=\int_\Sigma \dfrac{\sqrt{2\nu_3^2-1}}{\nu_3}\, d\Sigma,
\end{equation*}
whose Wulff shape is the hyperboloid $x^2+y^2-z^2=-1$. In non-parametric way, this functional is $\int_\Omega\sqrt{1-|Du|^2}$ which represents the area element of a spacelike surface in Lorentz-Minkowski space. Thus stationary surfaces  for volume-preserving variations are spacelike surfaces   with constant mean curvature. In this situation, the classification of cyclic CMC spacelike surfaces is known: \cite{lls,lo3,lo4}.

\section{Preliminaries}\label{s2}

In this section, we review some technicalities about the computation of the anisotropic mean curvature of a surface focusing in the Dirichlet energy. Let $\langle,\rangle$ denote both the Euclidean metric of $\r^3$ and   the induced one on a surface $\Sigma$ immersed in $\r^3$.   Let $X\colon\Sigma\to\r^3$ be an immersion of an oriented surface $\Sigma$.   For an axially symmetric   functional     \eqref{fr}, the anisotropic mean curvature $\Lambda$ is given by 
\begin{equation}\label{h}
\Lambda=\frac{h(v_1,v_1)}{\mu_1}+\frac{h(v_2,v_2)}{\mu_2},
\end{equation}
where $h$ is the second fundamental form of $\Sigma$ with respect to an orthonormal frame $\{v_1,v_2\}$ of $\Sigma$ and $\mu_1$ and $\mu_2$ are the principal curvatures of the Wulff shape \cite{kp1}. The principal directions $\{E_1,E_2\}$ are
$$E_1=e_3-\nu_3\nu,\quad E_2=\nu\times E_1,$$
where $e_3=(0,0,1)$ and  
$$\frac{1}{\mu_1}=(1-\nu_3^2)F''(\nu_3)+\frac{1}{\mu_2},\quad \frac{1}{\mu_2}=F-\nu_3F'(\nu_3).$$
For the Dirichlet energy \eqref{fu}, the principal curvatures $\mu_1$ and $\mu_2$  are
$$\frac{1}{\mu_1}=\frac{2}{\nu_3^3},\quad \frac{1}{\mu_2}=\frac{2}{\nu_3}.$$
Here it is understood that $\nu_3^2\not=1$ in order to have $E_1\not=0$. If $\nu_3^2\equiv 1$, then $\Sigma$ is a horizontal plane. In general, any plane is a stationary surface for $\Lambda=0$ because the second fundamental form $h$ is $0$ identically.

Notice that  $\{E_1,E_2\}$ is an orthogonal basis but no unitary. Thus the anisotropic mean curvature $\Lambda$ in \eqref{h} writes as 
$$\Lambda\nu_3^3=\frac{2h(E_1,E_1)}{|E_1|^2}+\frac{2h(E_2,E_2)\nu_3^2}{|E_2|^2}.$$
  Since $|E_1|^2=|E_2|^2=1-\nu_3^2$, this equation is equivalent to 
\begin{equation}\label{h2}
\Lambda\nu_3^3(1-\nu_3^2) =2(h(E_1,E_1)+ h(E_2,E_2)\nu_3^2) . 
\end{equation}
For the computation of $h(E_i,E_i)$ we use a parametrization $X=X(s,\theta)$ of $\Sigma$. Notice that the tangent plane is spanned by $\{X_s,X_\theta\}$. Here the subindices indicate the corresponding derivatives with respect to $s$ and $\theta$, respectively. Let  write $\{E_1,E_2\}$ in coordinates with respect to $\{X_s,X_\theta\}$,
\begin{equation}\label{e1e2}
\begin{split}
E_1&=c_{11} X_s+c_{12}X_\theta,\\
E_2&=c_{21}X_s+c_{22}X_\theta.
\end{split}
\end{equation}
Then
$$h(E_1,E_1)=c_{11}^2h(X_s,X_s)+2 c_{11}c_{12}h(X_s,X_\theta)+c_{12}^2h(X_\theta,X_\theta),$$
$$h(E_2,E_2)=c_{21}^2h(X_s,X_s)+2 c_{21}c_{22}h(X_s,X_\theta)+c_{22}^2h(X_\theta,X_\theta).$$
Let $g_{ij}$ be the coefficients of the first fundamental form of $X$ and $\nu=X_s\times X_\theta/|X_s\times X_\theta|$. Then 
\begin{equation*}
\begin{split}
h(X_s,X_s)&=\langle\nu,X_{ss}\rangle=\frac{\langle X_s\times X_\theta,X_{ss}\rangle}{\sqrt{\mbox{det}(g_{ij})}}:=\frac{h_{ss}}{\sqrt{\mbox{det}(g_{ij})}},\\
h(X_s,X_\theta)&=\langle\nu,X_{s\theta}\rangle=\frac{\langle X_s\times X_\theta,X_{s\theta}\rangle}{\sqrt{\mbox{det}(g_{ij})}}:=\frac{h_{s\theta}}{\sqrt{\mbox{det}(g_{ij})}},\\
h(X_\theta,X_\theta)&=\langle\nu,X_{\theta\theta}\rangle=\frac{\langle X_s\times X_\theta,X_{\theta\theta}\rangle}{\sqrt{\mbox{det}(g_{ij})}}:=\frac{h_{\theta\theta}}{\sqrt{\mbox{det}(g_{ij})}}.
\end{split}
\end{equation*}
Then Eq. \eqref{h2} writes as
\begin{equation}\label{h3}
\Lambda\nu_3^3(1-\nu_3^2) \sqrt{\mbox{det}(g_{ij})}=
2 (c_{11}^2hss+2c_{11}\,c_{12}hs\theta +c_{12}^2h\theta\theta)+2\nu_3^2 ( c_{21}^2hss+2c_{21}\,c_{22}hs\theta +c_{22}^2h\theta\theta).
\end{equation}

We now particularize the case that the surface is a surface of revolution about the $z$-axis. Since   $\nu_3\not=0$, the surface is locally a graph on the $xy$-plane at every point. Thus   the generating curve writes as  $\gamma(r)= (r,0,u(r))$,   $r\in I\subset\r^+$. The parametrization of the surface is  
$X(r,\theta)=(r\cos\theta,r\sin\theta,u(r))$, $r\in I$,   $\theta\in\r$. Now we can do a change of variables in the Euler-Lagrange equation $u_{xx}+u_{yy}=\Lambda/2$, or   computing directly \eqref{h3}. In this case, if $W=1+u'^2$, we have
$$\mbox{det}(g_{ij})=r\sqrt{W},\ \nu_3=\frac{1}{\sqrt{W}},\ h(E_1,E_1)=\frac{ru'u''}{W^2},\  h(E_2,E_2)=\frac{u'^3}{W}.$$
Then \eqref{h3} is  
$$u''+\frac{u'}{r}=\frac{\Lambda}{2}.$$
The solution of this equation is 
\begin{equation}\label{log}
u(r)=c_1\log(r)+\frac{\Lambda}{8}r^2+c_2,\quad c_1,c_2\in\r.
\end{equation}
Consequently, if $\Lambda=0$, $\Sigma$ is a horizontal plane ($c_1=0$) or $z=c_1\log(\sqrt{x^2+y^2})$. The Wulff shape appears by taking $c_1=0$ and $\Lambda>0$. 
\section{Proof of Theorem \ref{T1} }\label{s3}

Suppose that $\Sigma$ is a surface foliated by circles contained in planes parallel to the $xy$-plane. If all these planes coincide, then $\Sigma$ is a horizontal plane which is trivially an anisotropic minimal surface. From now, we discard this   situation.  Since  the planes of the foliation are   parametrized by $z=s$, $s\in\r$,   we can assume  that the curve formed by all centers of these circles is $s\mapsto (a(s),b(s),s)$, $s\in I\subset\r$ and $a,b\colon I\to\r$ are smooth functions. If $r(s)>0$ is the radius of each circle contained in  the plane $z=s$, then a parametrization of $\Sigma $ is  
\begin{equation}\label{para}
X(s,\theta)=(a(s),b(s),s)+r(s)(\cos\theta,\sin\theta,0),
\end{equation}
where $\theta\in \r$. Notice that $\Sigma$ is a surface of revolution about an axis parallel to the $z$-axis if and only if both functions $a=a(s)$ and $b=b(s)$ are constant. This in turn is equivalent to $a'(s)=b'(s)=0$ identically in $I$. First computations are \begin{equation*}
\begin{split}
X_s&=(a',b',1)+r'(\cos\theta,\sin\theta,0),\\
X_\theta&=r(-\sin\theta,\cos\theta,0),\\
 \mbox{det}(g_{ij})&=r^2(1+\left(r'+a'\cos \theta +b'\sin \theta \right)^2).
 \end{split}
 \end{equation*}
Let
\begin{equation}\label{w}
W=  1+\left(r'+a'\cos \theta +b'\sin \theta \right)^2.
\end{equation}
The Gauss map of $\Sigma$ is 
\begin{equation}\label{nn}
\nu=\frac{1}{\sqrt{W}}(-\cos \theta,-\sin \theta,r'+a'\cos \theta +b'\sin \theta ).
\end{equation}
In particular, 
\begin{equation}\label{nu3}
\nu_3=\frac{1}{\sqrt{W}}(r'+a'\cos \theta +b'\sin \theta).
\end{equation}
The basis $\{E_1,E_2\}$ is  
\begin{equation*}
\begin{split}
E_1&=\frac{1}{W}( \cos \theta \left(r'+a'\cos \theta +b'\sin \theta \right) , \sin \theta \left(r'+a'\cos \theta +b'\sin \theta \right) ,1),\\
E_2&=\frac{1}{\sqrt{W}}(-\sin \theta,\cos \theta,0).
\end{split}
\end{equation*}
In terms of the basis $\{X_s,X_\theta\}$, we have 
$$\begin{array}{lll}
&c_{11}=\dfrac{1}{W},&c_{12}=\dfrac{a'\sin\theta-b'\cos\theta}{rW},\\
&c_{21}=0,&c_{22}=\dfrac{1}{r\sqrt{W}}.
\end{array}$$
For the computation of $h(E_1,E_1)$ and $h(E_2,E_2)$, we find
\begin{equation*}
\begin{split}
h_{ss}&= -r(r''+a''\cos \theta +b''\sin \theta),\\
h_{s\theta}&=0,\\
h_{\theta\theta}&=r^2.
\end{split}
\end{equation*}
 Notice that $1-\nu_3^2=1/W$. We have 
\begin{equation*}
\begin{split}
h(E_1,E_1)&=\frac{\left(a'\sin \theta -b'\cos \theta \right)^2-r \left(r''+a''\cos \theta +b''\sin \theta\right)}{W^2},\\
h(E_2,E_2)&=\frac{1}{W}.
\end{split}
\end{equation*}
Using \eqref{nu3}, equation \eqref{h3} becomes
\begin{equation*}
 \Lambda r(r'+a'\cos \theta +b'\sin \theta)^3    =2(r'+a'\cos \theta +b'\sin \theta)^2  +2 \left((a'\sin\theta-b'\cos\theta)^2-r  (r''+a''\cos\theta +b''\sin\theta   )\right).
\end{equation*}
We write this identity as   an equation of type 
$$\sum_{n=0}^3 A_n(s)\cos(n\theta)+B_n(s)\sin(n\theta) =0.$$
Since the functions $\{\cos(n\theta),\sin(n\theta):0\leq n\leq 3\}$ are linearly independent, all coefficients $A_n$ and $B_n$ must be $0$ identically. For $n=3$ we have 
\begin{equation*}
\begin{split}
A_3&=\frac{1}{4} \Lambda   ra' (a'^2-3b'^2),\\
B_3&=\frac{1}{4} \Lambda ^2 r    b'  (3 a'^2-  b'^2).
\end{split}
\end{equation*}
\begin{enumerate}
\item Case $\Lambda\not=0$. Since $r\not=0$, both equations  $A_3=0$ and $B_3=0$  imply  $a'=b'=0$  identically. This proves that the surface is rotational, obtaining (1) of Thm. \ref{T1}.

\item Case $\Lambda=0$. Now Eq. \eqref{h3} is of degree $1$, namely,  
\begin{equation}\label{k}
 (2ra''-4a'r')\cos\theta+(2r b''-4b'r')\sin\theta-2(a'^2+b'^2+r'^2-rr'')=0.
\end{equation}
This gives    
\begin{equation}\label{rie}
\begin{split}
ra''-2a'r'&=0,\\
rb''-2b'r'&=0,\\
a'^2+b'^2+r'^2-rr''&=0.
\end{split}
\end{equation}
If $a'=b'=0$ simultaneously, then the surface is rotational and the equation for $r$ is $r'^2-rr''=0$. Its solutions is $r(s)=r_0>0$ (horizontal plane) or  $r(s)=c_2 e^{c_1s}$, $c_2>0$.  Then $z=s=\frac{1}{c_1}\log(r)$ and this gives the first part of (2) in Thm. \ref{T1}: this coincides with the expression given in  \eqref{log}.  If $a'$ or $b'$ is not $0$ identically, an integration of the   first two equations implies that there exist $\lambda,\mu\in\r$ such that 
$a'=\lambda r^2$ and $b'=\mu r^2$ and $\lambda^2+\mu^2\not=0$. In case that $\lambda$ or $\mu$ is $0$, then $a$ or $b$ is a constant function. This does not affect in the rest of computations after the corresponding changes. Substituting into the third equation of \eqref{rie} we obtain 
\begin{equation}\label{rr}
(\lambda^2+\mu^2)r^4+r'^2-rr''=0.
\end{equation}
We solve \eqref{rr}. For this, let $u=r'$ and $u=u(r)$. Then $r''(s)=uu'$ and \eqref{rr} writes 
$$ruu'=(\lambda^2+\mu^2)r^4+u^2.$$
Hence we solve $u$ obtaining 
$$r'=\pm r\sqrt{(\lambda^2+\mu^2)r^2+c_1},$$
where $c_1\in\r$ is an integration constant. Then  
$$\int\frac{dr}{r\sqrt{(\lambda^2+\mu^2)r^2+c_1}}=\pm s+c_2,\quad c\in\r.$$
The solution of this equation depends on the sign of $c_1$. Without loss of generality, we can assume $c_2=0$ because this only gives a translation in the domain $I$ of the functions $a$, $b$ and $r$ and also, a vertical translation of the surface because of the third coordinate in the parametrization \eqref{para}.
\begin{enumerate}
\item Case $c_1<0$. Then 
$$r(s)=\frac{c}{\sqrt{\lambda^2+\mu^2}\cos(cs)},$$
for some $c\in\r$. From this value of $r$, we obtain $a$ and $b$ obtaining \eqref{sol1}.

\item Case $c_1=0$. Then 
$$r(s)=\frac{1}{\sqrt{\lambda^2+\mu^2}s+c},$$
for some $c\in\r$. Hence, we have  $a$ and $b$ as in \eqref{sol2}.
\item Case $c_1>0$. Then 
$$r(s)=\frac{c}{\sqrt{\lambda^2+\mu^2}\sinh(cs)},$$
for some $c\in\r$. From this value of $r$, we obtain $a$ and $b$ obtaining \eqref{sol3}.
\end{enumerate}

\end{enumerate}
Let $\c(s)=(a(s),b(s),s)$ be the curve of centers of the circles of the foliation.  Since $\c'(s)=(\lambda r^2,\mu r^2,1)$, then $\c'(s)$ is orthogonal to the horizontal constant vector $(-\mu,\lambda,0)$. This proves that $\c(s)$ lies contained in the vertical plane of equation $-\mu x+\lambda y=0$. This plane is a plane of symmetry of the surface.   The surface  is embedded because the curve of centers $\c(s)$ is a graph on the $z$-axis.
This finishes the proof of Thm. \ref{T1}.

\begin{remark}Comparing the above computations with the isotropic case, $F\equiv 1$, the first two equations of \eqref{rie} coincide. The difference is the third equation of \eqref{rie} which, in the isotropic case is $ 1+a'^2+b'^2+r'^2-rr''=0$. For this equation, if $a'=b'=0$,  the solution is the plane $r(s)=r_0>0$ and the catenoid $r(s)=c_1\cosh(\frac{s+c_2}{c_1})$. If $a'^2+b'^2\not=0$, then the solutions for $a$, $b$ and $r$ are given in terms of elliptic integrals obtaining the corresponding parametrizations of  the Riemann  minimal examples. The non-presence of  `$1$' in the third equation of \eqref{rie} allows an explicit integration of the equation for the Dirichlet energy. 
\end{remark}
 
We finish this section obtaining geometric properties of the non-rotational surfaces of Thm. \ref{T1}. Without loss of generality, we assume $\lambda>0$.  After a rotation with respect to the $z$-axis, which does not affect to the value of $\Lambda$, and up to   translations in $\r^3$, we can suppose $\mu=0$ and $b(s)=0$.  The study of the surfaces will be done according to each one of the types of item (2): surfaces of types I, II and III   will be described in Thms. \ref{t5},  \ref{t6} and \ref{t7}, respectively.

\begin{theorem}[Type I]\label{t5} Let $\Sigma_{\lambda,c}$ be a surface parametrized by 
\begin{equation}\label{rie1}
X(s,\theta)=\left(\frac{c}{\lambda} \frac{ \sin(cs)+\cos\theta}{ \cos (cs)},  \frac{ c}{ \lambda \cos (cs)}\sin\theta,s\right),\quad s\in I:=(-\frac{\pi}{2c},\frac{\pi}{2c}),\theta\in\r.
\end{equation}
Here 
$$a(s)=\frac{c}{\lambda}\tan(cs),\quad r(s)=\frac{c}{\lambda\cos(cs)},\quad c>0.$$
\begin{enumerate}
\item The surface $\Sigma_{\lambda,c}$ lies contained in the horizontal slab $\{(x,y,z)\in\r^3:-\frac{\pi}{2c}< z< \frac{\pi}{2c}\}$. 
\item The surface $\Sigma_{\lambda,c}$ is symmetric about the plane $ y=0$ and also  under the $180^0$-rotation about the $y$-axis.
\item The orthogonal projections onto the $xy$-plane of any two circles of the foliation   overlap.
\item The intersection of $\Sigma_{\lambda,c}$ with the planes $z=\pm \frac{\pi}{2c}$ are the  straight-lines   of equation $\{x=0,z=\pm\frac{\pi}{2c}\}$.
\item The surface $\Sigma_{\lambda,c}$  can be extended by vertical translations $(x,y,z)\mapsto (x,y,\frac{(2k+1)\pi}{2c})$, $k\in\mathbb{Z}$,  obtaining a   stationary surface $\widetilde{\Sigma}_{\lambda,c}$. This surface  is foliated by circles and lines in parallel planes. The lines only occur at the horizontal planes   $z= \frac{(2k+1)k\pi}{2c}$  and the circles in the rest of horizontal planes.
\item The horizontal planes  $z=\pm   \frac{(2k+1) \pi}{2c}$  are ends of $\widetilde{\Sigma}_{\lambda,c}$ .

\end{enumerate}
\end{theorem}

\begin{proof}
The surface $X(s,\theta)$ is depicted in Fig. \ref{fig1}.
\begin{enumerate}  
\item It is a consequence that the domain of the $s$-variable is   $I=(-\frac{\pi}{2c},\frac{\pi}{2c})$. 
\item The first symmetry is immediate because $\c$ is contained in the plane $y=0$. On the other hand,   if $R\colon\r^3\to\r^3$ given by      $R(x,y,z)=(-x,y,-z)$ is the $180^0$ rotations about the $y$-axis, it holds $R\circ X(s,\theta)=X(-s,\pi-\theta)$ for all $s,\theta$, proving the result.
\item  Two circles at heights $s_1$ and $s_2$ overlap if and only if  $|a(s_1)-a(s_2)|<r(s_1)+r(s_2)$. From the expressions of $a(s)$ and $r(s)$, this inequality writes as
$$ |\tan(cs_1)-\tan(cs_2)|<\frac{1}{\cos (cs_1)}+\frac{1}{\cos (cs_2)},$$
which is trivially true.

\item Let $L_1$ be the straight-line of equation $\{x=0,z= \frac{\pi}{2c}\}$. We prove that any point of $L_1$ is limit of points of $\Sigma_{\lambda,c}$. Given an arbitrary point $(0,y,\frac{\pi}{2c})\in L_1$, we have 
$$\lim_{s\to\pi/(2c)}X(s,\pi-\frac{y}{r(s)})=\lim_{s\to\pi/(2c)}\left(a(s)-r(s)\cos(\frac{y}{r(s)}),r(s)\sin(\frac{y}{r(s)}),s\right)=(0,y,\frac{\pi}{2c}).$$
By the symmetry with respect to the $180^0$-rotation about the $y$-axis, we obtain that the line $\{x=0,z=-\frac{\pi}{2c}\}$ is also limit of points of $\Sigma_{\lambda,c}$. 
\item  Since the third coordinate $z=z(x,y)$ of $\Sigma_{\lambda,c}$ is harmonic, the   property is consequence of  the Schwarz reflection principle.
\item Each surface $\Sigma_{\lambda,c}$ is asymptotic to the family of   horizontal  planes $z=\pm \frac{\pi}{2c}$, $k\in\mathbb{Z}$, because fixing $\theta$, $\theta\not=\pi$, we have $\lim_{s\to\pm\pi/(2c)}X(s,\theta)=(\infty,\infty,\pm\frac{\pi}{2c})$. The result for  $\widetilde{\Sigma}_{\lambda,c}$ follows by vertical translations.
\end{enumerate}
\end{proof}

\begin{figure}[hbtp]
\centering
\includegraphics[width=.45\textwidth]{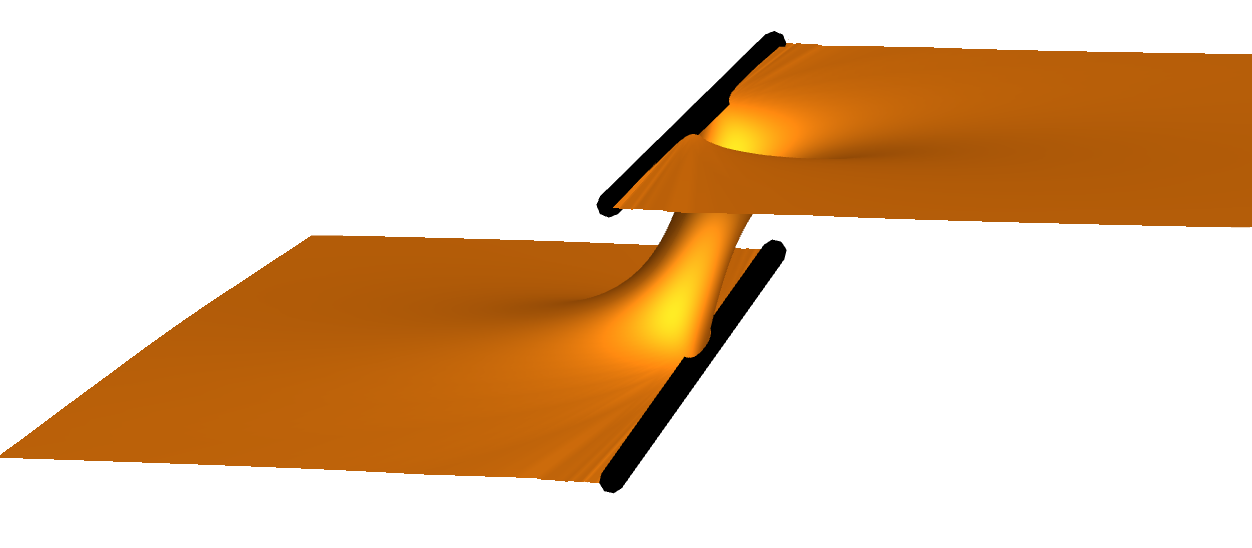}\qquad \includegraphics[width=.35\textwidth]{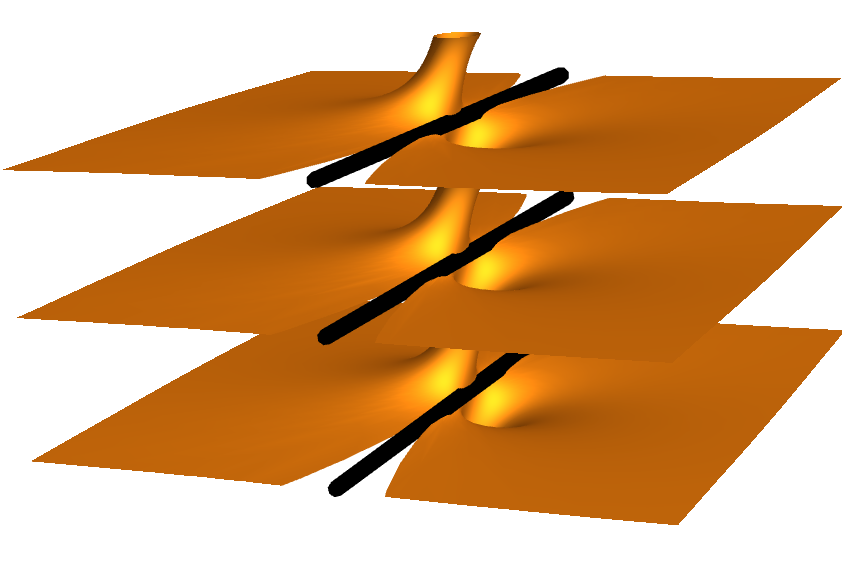}
\caption{Surfaces $\Sigma_{\lambda,c}$ and $\widetilde{\Sigma}_{\lambda,c}$ of type I. Here $\lambda=2$ and $c=1$. Left: the surface $\Sigma_{\lambda,c}$, \eqref{rie1}. Right: the   surface $\widetilde{\Sigma}_{\lambda,c}$ after vertical translations of $\Sigma_{\lambda,c}$. Horizontal  (black) straight-lines are contained on the surface.}\label{fig1}
\end{figure}


\begin{theorem}[Type II] \label{t6}
Let $\Sigma_{\lambda,c}$ be a surface parametrized by 
\begin{equation}\label{rie2}
X(s,\theta)=\left(\frac{-1+\cos\theta}{\lambda s+c},\frac{\sin\theta}{\lambda s+c},s\right),\quad s\in I:=(-\frac{c}{\lambda},\infty),\theta\in\r.
\end{equation}
Here 
$$a(s)=-\frac{1}{\lambda s+c} ,\quad r(s)=\frac{1}{\lambda s+c},\quad c>0.$$
\begin{enumerate}
\item The surface $\Sigma_{\lambda,c}$ lies contained in the half-space   $\{(x,y,z)\in\r^3:  z> -\frac{c}{\lambda}\}$. 
\item The surface $\Sigma_{\lambda,c}$ is asymptotic to the $z$-axis.
\item All circles of the foliation touch the $z$-axis at exactly only point. 
\item The intersection of $\Sigma_{\lambda,c}$ with the plane  $z=-\frac{c}{\lambda}$ is  the  straight-line $L$ of equation $\{x=0,z=-\frac{c}{\lambda}\}$.
\item The surface $\Sigma_{\lambda,c}$  can be extended by the $180^0$-rotation about  $L$  obtaining a   stationary surface $\widetilde{\Sigma}_{\lambda,c}$. This surface  is foliated by circles and lines in parallel planes. The line only occurs at the horizontal plane   $z= -\frac{c}{\lambda}$  and the circles in the rest of horizontal planes.
\item The horizontal plane   $z=-\frac{c}{\lambda}$  is an end of $\widetilde{\Sigma}_{\lambda,c}$ .

\end{enumerate}
\end{theorem}

\begin{proof}
The surface $X(s,\theta)$ is depicted in Fig. \ref{fig2}.
\begin{enumerate}

\item It is a consequence that the domain of the $s$-variable is   $I=(-\frac{c}{\lambda},\infty)$. 
\item For any $\theta\not=0$, $\lim_{s\to\infty}X(s,\theta)=(0,0,\infty)$.
  \item For all $s\in I$, we have $X(s,0)=(0,0,s)$.

\item Let $(0,y,\frac{\pi}{2c})$ be   an arbitrary point of $L$. Then  
$$\lim_{s\to-\frac{c}{\lambda}}X(s, \frac{y}{r(s)})=(0,y,-\frac{c}{\lambda}).$$
\item  It is consequence of  the Schwarz reflection principle about the line $L$ and that the coordinate function $z=z(x,y)$ is harmonic on the surface. 
\item Immediate.
\end{enumerate}
\end{proof}

\begin{figure}[hbtp]
\centering
\includegraphics[width=.35\textwidth]{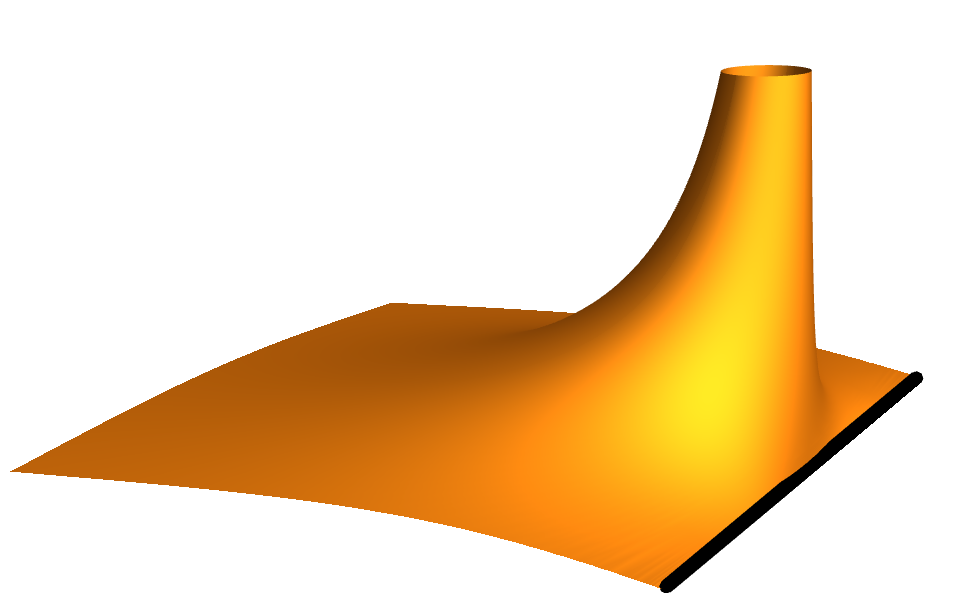}\qquad \includegraphics[width=.35\textwidth]{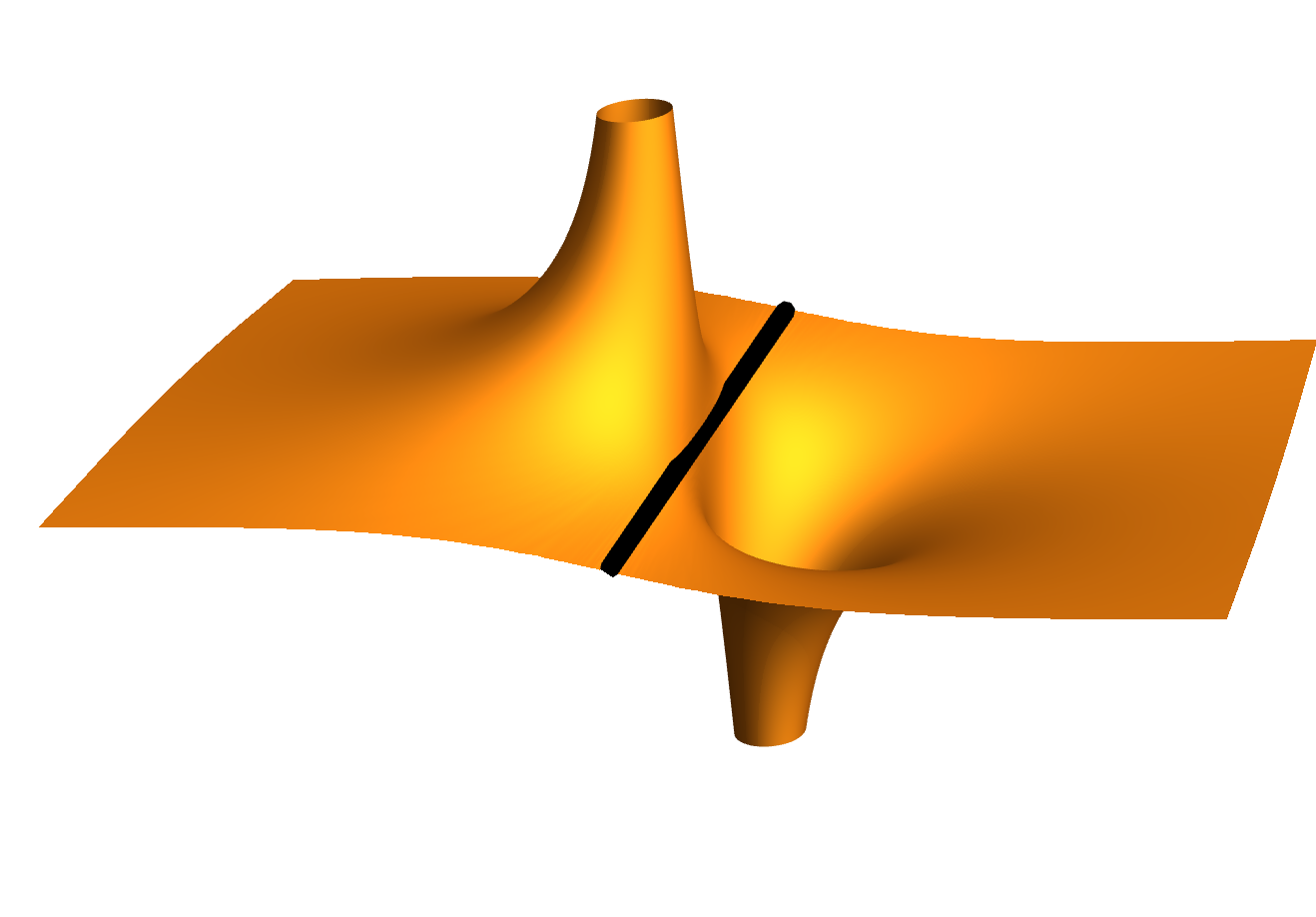}
\caption{Surfaces $\Sigma_{\lambda,c}$ and $\widetilde{\Sigma}_{\lambda,c}$ of type II. Here $\lambda=1$ and $c=0$. Left: the  surface $\Sigma_{\lambda,c}$, \eqref{rie2}. Right: the   surface $\widetilde{\Sigma}_{\lambda,c}$ after the $180^0$ rotation about $L$ of $\Sigma_{\lambda,c}$. }\label{fig2}
\end{figure}

\begin{remark} Notice that the surface \eqref{rie2} contains the $z$-axis, which it is a vertical line. However, we cannot assert that $\Sigma_{\lambda,c}$ may be extended by reflections across the $z$-axis because the coordinate function $y=y(x,z)$ on the surface is not harmonic in general.
\end{remark}


\begin{theorem}[Type III] \label{t7}
Let $\Sigma_{\lambda,c}$ be a surface parametrized by 
\begin{equation}\label{rie3}
X(s,\theta)=\left(\frac{c}{\lambda} \frac{ \cosh(cs)+\cos\theta}{ \sinh (cs)},  \frac{ c}{ \lambda \sinh (cs)}\sin\theta,s\right),\quad s\in I:=(0,\infty),\theta\in\r.
\end{equation}
Here 
$$a(s)=\frac{c}{\lambda}\coth(cs),\quad r(s)=\frac{c}{\lambda\sinh(cs)},\quad c>0.$$
\begin{enumerate}
\item The surface $\Sigma_{\lambda,c}$ lies contained in the half-space   $\{(x,y,z)\in\r^3:  z> 0\}$. 
\item The surface $\Sigma_{\lambda,c}$ is asymptotic to the vertical line $\{x=\frac{c}{\lambda},y=0\}$.
 \item The intersection of $\Sigma_{\lambda,c}$ with the plane  $z=0$ is  the  $y$-axis.
\item The surface $\Sigma_{\lambda,c}$  can be extended by the $180^0$-rotation about  the $y$-axis  obtaining a   stationary surface $\widetilde{\Sigma}_{\lambda,c}$. This surface  is foliated by circles and lines in parallel planes. The line only occurs at the horizontal plane   $z= 0$  and the circles in the rest of horizontal planes.
\item The horizontal plane   $z=0$  is an end of $\widetilde{\Sigma}_{\lambda,c}$ .

\end{enumerate}
\end{theorem}

\begin{proof}
The surface $X(s,\theta)$ is depicted in Fig. \ref{fig3}.
\begin{enumerate}

\item It is a consequence that the domain of the $s$-variable is   $I=(0,\infty)$. 
\item We have $\lim_{s\to\infty}X(s,\theta)=(\frac{c}{\lambda},0,\infty)$.
 
\item Let $(0,y,0)$ be   an arbitrary point of the $y$-axis. Then  
$$\lim_{s\to 0}X(s, \pi-\frac{y}{r(s)})=(0,y,0).$$
\item  It is consequence of  the Schwarz reflection principle about the $y$-axis. 
\item Immediate.
\end{enumerate}
\end{proof}

\begin{remark} Notice that each one of the family of surfaces of (2) of Thm. \ref{T1} depends only on one parameter because after a dilation by a positive ratio, one of the constants $\lambda$ or $c$ can be eliminated.
\end{remark}

\begin{figure}[hbtp]
\centering
 \includegraphics[width=.35\textwidth]{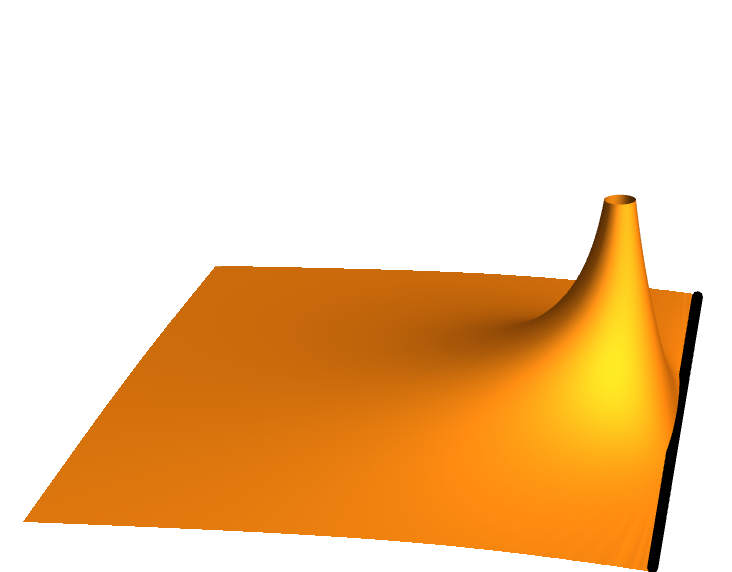}\qquad \includegraphics[width=.35\textwidth]{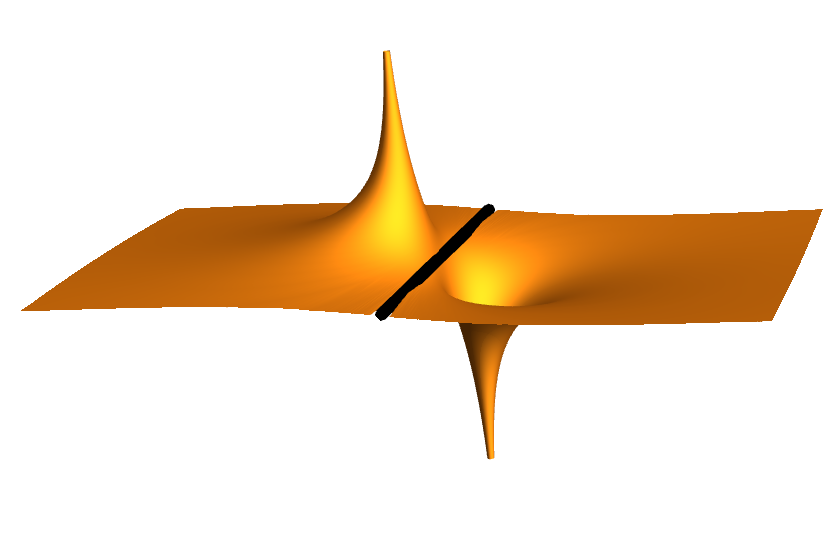}
 \caption{Surfaces $\Sigma_{\lambda,c}$ and $\widetilde{\Sigma}_{\lambda,c}$ of type III. Here $\lambda=1$ and $c=0$. Left: the  surface $\Sigma_{\lambda,c}$, \eqref{rie3}. Right: the   surface $\widetilde{\Sigma}_{\lambda,c}$ after the $180^0$ rotation about $L$ of $\Sigma_{\lambda,c}$.}\label{fig3}
\end{figure}
 
 In Fig. \ref{fig4} we show the intersection of the three types of surfaces  with the vertical plane $y=0$. 
  
  \begin{figure}[hbtp]
\centering
\includegraphics[width=.42\textwidth]{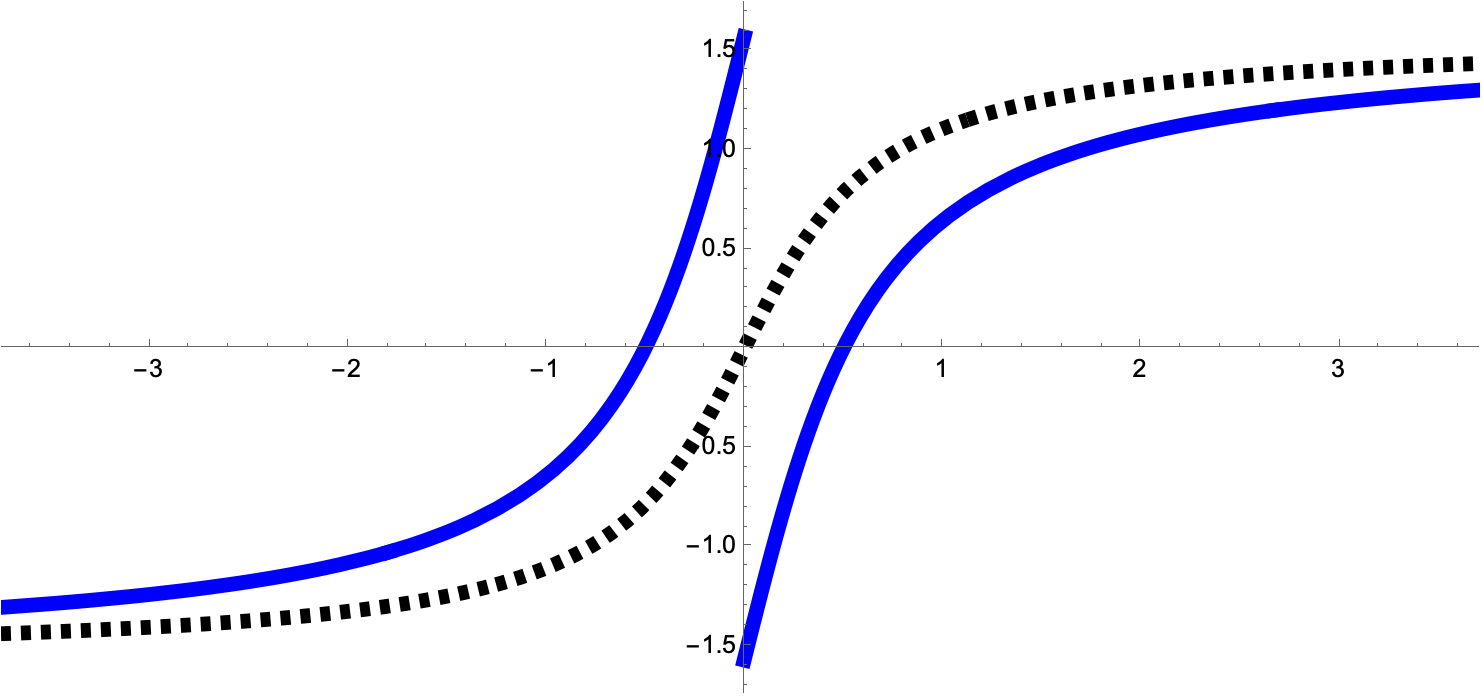} \quad
\includegraphics[width=.27\textwidth]{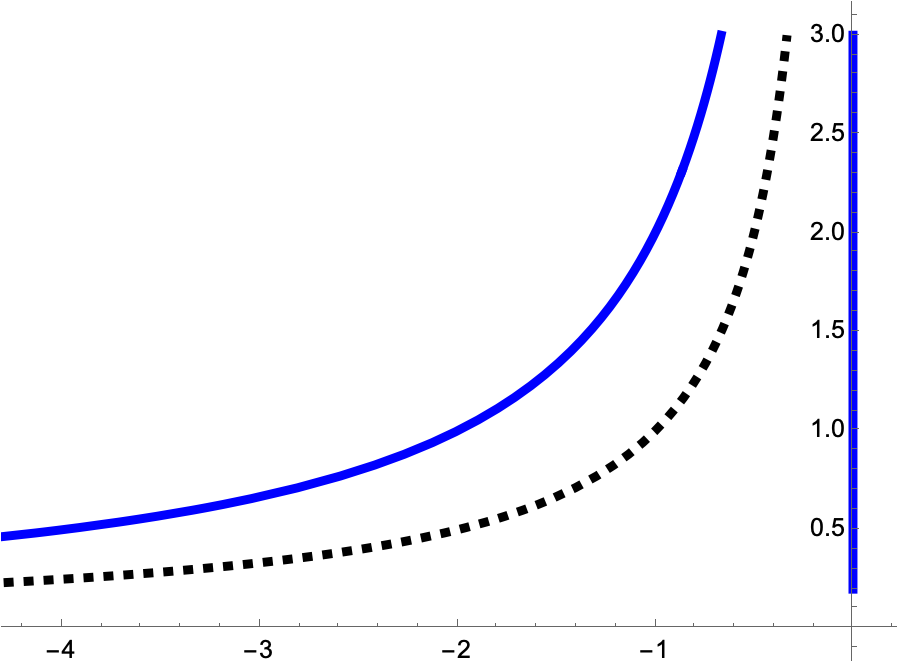} \quad
\includegraphics[width=.2\textwidth]{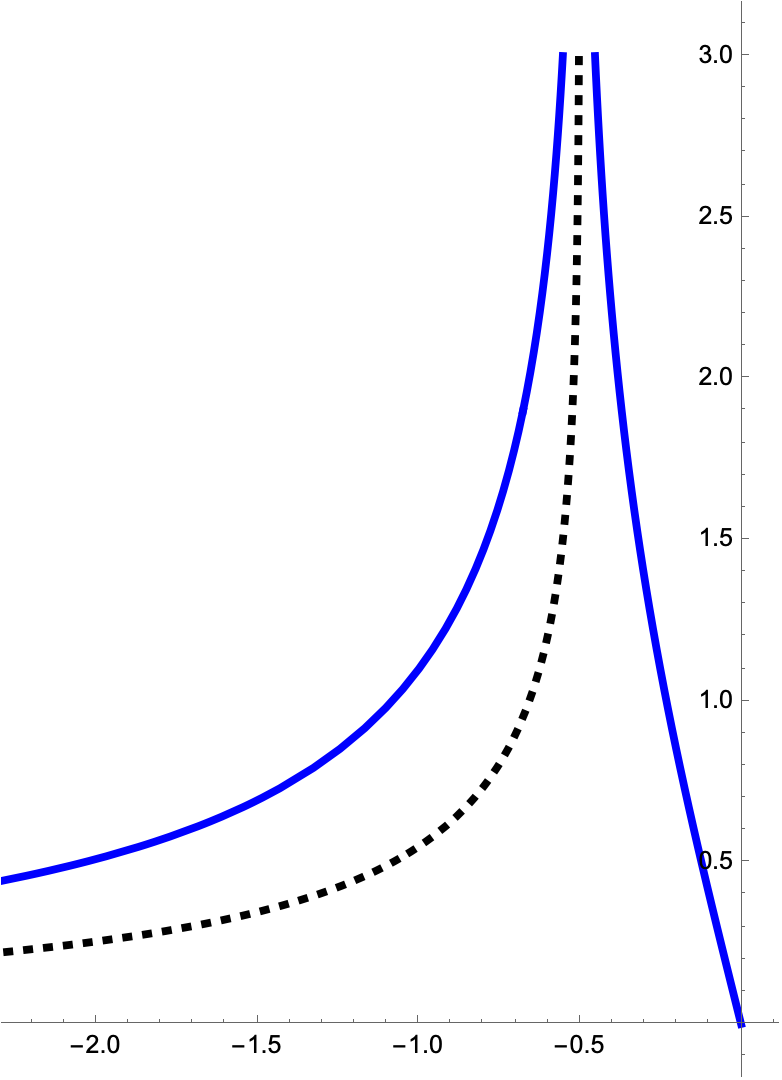} 
\caption{Cross-section of the surfaces \eqref{rie1} (left), \eqref{rie2} (middle) and \eqref{rie3} (right) with the plane $y=0$. The dashed line is the curve of centers $\c(s)$ of the circles. The blue lines represents the intersection of the surface with the plane $y=0$.  }\label{fig4}
\end{figure}
\section{Proof of Theorem \ref{T2}}\label{s4}

 The separate the proof of Thm. \ref{T2} in two subsections.  We   first prove that the planes of the foliation must be parallel. Next, we will show   that these planes are parallel to the $xy$-plane.

\subsection{Claim: the planes of the foliation are parallel}

In this part, we will assume that $\Sigma$ is a CAMC surface with $\Lambda$ a constant not necessarily $0$. Once we prove this claim, Theorem \ref{T3} is also proved thanks to Thm. \ref{T1}.

  The proof  is by contradiction. Suppose that the planes are not parallel. Consider $s$ the parameter of the foliation and let $\Gamma\colon I\subset\r\to\r^3$, $\Gamma=\Gamma(s)$,  be a smooth curve which is orthogonal to each $s$-plane. Without loss of generality, we can suppose that $s$ is the arc-length parameter of $\Gamma$. Since the planes of the foliation are not parallel, then the curve $\Gamma$ is not a straight-line, in particular, the curvature $\kappa$ of $\Gamma$ is non-zero. This also implies the existence of a Frenet frame    $\{\t,\n,\b\}$, where $\t(s)=\Gamma'(s)$ and $\n$ and $\b$ are the normal and binormal vector of $\Gamma$, respectively. The Frenet equations of $\Gamma$ are
\begin{equation*}
\begin{split}
\t'&=\kappa\n,\\
\n'&=-\kappa\t+\tau\b,\\
\b'&=-\tau\n,
\end{split}
\end{equation*}
where  $\tau$ is the torsion of $\Gamma$. According to $\Gamma$,  a parametrization of  $\Sigma$ is
$$X(s,\theta)=\c(s)+r(s)(\cos\theta\,\n(s)+\sin\theta\,\b(s)),$$
where  $\c(s)$ and $r(s)>0$ are the center and the radius, respectively, of each circle of the foliation. Let us express the derivative $\c'(s)$ in terms of the Frenet frame, 
\begin{equation}\label{cc}
\c'=\alpha\t+\beta\n+\gamma\b,
\end{equation}
where $\alpha,\beta,\gamma$ are smooth functions on $s$. 

We now compute all terms in Eq. \eqref{h3}.    Using \eqref{cc}, the tangent basis $\{X_s,X_\theta\}$ is expressed in coordinates with respect to the Frenet frame by  
\begin{equation*}
\begin{split}
X_s&=(\alpha - r \kappa\cos \theta  ,  r'\cos \theta-r\tau \sin\theta    +\beta , r'\sin \theta  +r \tau\cos \theta   +\gamma ),\\
X_\theta&=(0,-r\sin \theta   ,r\cos \theta  ).
\end{split}
\end{equation*}
The unit normal is 
$$\nu=\frac{1}{\sqrt{W}}\left(r'+\beta\cos \theta  +\gamma \sin \theta  ,\cos \theta ( r \kappa\cos \theta -\alpha ),\sin \theta (r \kappa\cos \theta  -\alpha )\right),$$
where
$$W= (r'+\beta\cos \theta  +\gamma\sin \theta   )^2+\cos ^2\theta (r \kappa\cos \theta-\alpha   )^2+\sin ^2\theta ( r \kappa\cos \theta-\alpha  )^2.$$
Notice that $\mbox{det}(g_{ij})=r^2W$. We now express $e_3$  in coordinates with respect to the Frenet frame. Let $e_3=(e_{11},e_{22},e_{33})$. We point out that these coordinates are not constant in general. Then
\begin{equation}\label{ee}
E_1=(e_{11},e_{22},e_{33})-\langle\nu,e_3\rangle\nu,\quad E_2=\nu\times E_1.
\end{equation}
 Following \eqref{e1e2}, let   write $\{E_1,E_2\}$ in coordinates with respect to $\{X_s,X_\theta\}$.  
The calculations for $c_{ij}$ are some tedious and Mathematica can help. For example,   expressions of $c_{21}$ and $c_{22}$ are simple, namely, 
 
\begin{equation*}
\begin{split}
 c_{21}&=\frac{1}{\sqrt{W}}(e_{22}\sin\theta-e_{33}\cos\theta),\\
 c_{22}&=\frac{1}{r\sqrt{W}}\left(  (\alpha-r\kappa\cos\theta)e_{11}+(\beta-r\tau\sin\theta+r'\cos\theta)e_{22}+(\gamma+r\tau\cos\theta+r'\sin\theta)e_{33} \right).
 \end{split}
\end{equation*}
Using these $c_{ij}$'s, equation \eqref{h2} can be written as
\begin{equation}\label{ab}
\sum_{n=0}^{10}A_n(s)\cos(n\theta)+B_n(s)\sin(n\theta)=0,
\end{equation}
where it is possible to compute all coefficients $A_n$ and $B_n$. We begin for $n=10$, obtaining
\begin{equation*}
\begin{split}
A_{10}&=-\frac{1}{512} e_{22}(e_{22}^4-10 e_{22}^2 e_{33}^2+5 e_{33}^4)\Lambda r^6\kappa^6,\\
B_{10} &=- \frac{1}{512}e_{33}(5e_{22}^4-10 e_{22}^2 e_{33}^2+ e_{33}^4)\Lambda r^6\kappa^6.\\
\end{split}
\end{equation*}
If $\Lambda\not=0$, the equations $A_{10}=0$ and $B_{10}=0$ imply $e_{22}=e_{33}=0$. In case $\Lambda=0$, the coefficients for $n=9$ are
\begin{equation*}
\begin{split}
A_{9}&=-\frac{1}{128}  (e_{22}^4-6 e_{22}^2 e_{33}^2+  e_{33}^4)  r^5\kappa^5,\\
B_{9} &= -\frac{1}{32} e_{22}e_{33}(e_{22}^2-e_{33}^2)   r^5\kappa^5.\\
\end{split}
\end{equation*}
Then $A_9=0$ and $B_9=0$ gives   $e_{22}=e_{33}=0$ again. Thus, regardless if $\Lambda$ is $0$ or not,   we conclude that the vector $e_3$ writes as $e_3=(e_{11},0,0)$. Since $e_3$ is unitary, then  $e_{11}=\pm 1$. This implies $\t=\pm e_3$ is constant and, consequently  $\Gamma$ is a straight-line, which it is a contradiction.

\subsection{Claim: if  the planes of the foliation are parallel,   then they are parallel to the $xy$-plane}\label{32}

Once we have proved that all planes of the foliation are parallel, the next step is to prove that these planes are parallel to the $xy$-plane. Now, we are assuming $\Lambda=0$.  Let $\mathcal{B}=\{v_1,v_2,v_3\}$ be an orthonormal basis of $\r^3$ such that the planes of the foliation are orthogonal to $v_3$. We parametrize $\Sigma$ in coordinates with respect to $\mathcal{B}$.  The expression of a parametrization $X$ is as in \eqref{para} with the convention that  the coordinates of $X$ and $\nu$ are with    respect to $\mathcal{B}$. Then the vector $e_{3}$ as well as of $E_1$ and $E_2$ in \eqref{ee} must be written in coordinates with respect to $\mathcal{B}$.  We follow the same computations. 

After a rotation about $v_3$, we can assume that $e_3$ is contained in the $(v_1,v_3)$-plane. Le $e_3=(e_{11},0,e_{33})$.  The claim    is proved if we see that $e_{11}=0$. 

Let $W$ and $\nu$ as in  \eqref{w} and \eqref{nn}, respectively. Again $E_1=(e_{11},0,e_{33})-\langle\nu,e_3\rangle\nu$ and $E_2=\nu\times E_1$, where 
$$\nu_3=\langle\nu,e_3\rangle=\frac{1}{\sqrt{W}}(-e_{11}\cos\theta+e_{33}(r'+a'\cos\theta+b'\sin\theta)).$$

The coefficients $c_{ij}$ in \eqref{e1e2} are
\begin{equation*}
\begin{split}
c_{11}&=\frac{1}{W}(e_{33}+(e_{11} \cos \theta)  (a'\cos \theta  +b'\sin \theta  +r' )),\\
c_{12}&=\frac{1}{2rW}(a'  (-2b'  e_{11} \cos \theta+3r'  -e_{11} \sin (2 \theta))+2 e_{33} \sin \theta ) \\
&+b' \left(r' (e_{11} (\cos (2 \theta)-3)-2 e_{33} \cos \theta\right)-2 e_{11} \sin \theta b'^2-2e_{11}\sin\theta  (r'^2+1 )  ),\\
c_{21}&=\frac{1}{\sqrt{W}}e_{11} \sin \theta,\\
c_{22}&=\frac{1}{r\sqrt{W}}(e_{11} a' +r' e_{11} \cos \theta+e_{33}).
\end{split}
\end{equation*}

Equation \eqref{h2} writes as 
$$\sum_{n=0}^5 A_n(s)\cos(n\theta)+B_n(s)\sin(n\theta)=0.$$

We have 
$$A_5=-\frac{1}{8} e_{11}^2 r \left(2 b' b'' \left(e_{11} e_{33}-(e_{33}^2-1) a'\right)+a'' \left(-2 e_{11} e_{33} a'+(e_{33}^2-1) a'^2-\left(e_{33}^2-1\right) b'^2+e_{11}^2\right)\right),$$
$$B_5=\frac{1}{8} e_{11}^2 r \left(2 a'' b' \left(e_{11} e_{33}-(e_{33}^2-1) a'\right)-b'' \left(-2 e_{11} e_{33} a'+(e_{33}^2-1) a'^2-(e_{33}^2-1) b'^2+e_{11}^2\right)\right).$$
From $A_5=0$ and $B_5=0$, we can do linear combinations in order to eliminate $a''$ and $b''$. Since the arguments are interchangeable between $a''$ and $b''$, we  eliminate $b''$ obtaining
\begin{equation}\label{pa}
e_{11}^4 r^2 a'' \left(\left(e_{11}-(e_{33}-1) a'\right)^2+(e_{33}-1)^2 b'^2\right) \left(\left(e_{11}-(e_{33}+1) a'\right)^2+(e_{33}+1)^2 b'^2\right)=0.
\end{equation}
If $e_{11}=0$, the result is proved. From now, we will assume $e_{11}\not=0$ and we will arrive to a contradiction. The above identity \eqref{pa} gives the following discussion of cases. 
\begin{enumerate}
\item Case $a''=0$. If $a(s)=a_1s+a_0$, then 
$$A_5=-\frac14e_{11}^3rb'b''(e_{11}a_1+e_{33}).$$
\begin{enumerate}
\item Subcase $b''=0$. Let $b(s)=b_1s+b_0$. We obtain 
\begin{equation*}
\begin{split}
A_4&=\frac{1}{4} e_{11}^3((a_1^2-b_1^2-1)e_{11}+2a_1 e_{33}) (r r''+r'^2),\\
B_4&=\frac{1}{2} b_1e_{11}^3(a_{1}e_{11}+e_{33}) (r r''+r'^2 ).
\end{split}
\end{equation*}
\begin{enumerate}
\item Case  $rr''+r'^2=0$. The solutions are $r(s)=r_0>0$ and $r(s)=c_2\sqrt{2s+c_2}$, $c_1>0$.
\begin{enumerate}
\item Case   $r(s)=r_0>0$.   Then 
$$B_2=2 b_1e_{11} (a_1 e_{11}+e_{33}) \left(1+b_1^2-a_1e_{11}e_{33}+e_{33}^2(a_1^2-1)\right).$$
Equation $B_2=0$ gives three cases.
\begin{itemize}
\item If $b_1=0$, then $A_2=(e_{11}-a_1 e_{33})^2(e_{33}+a_1 e_{11})(e_{33}+a_1 e_{11}-1)$. If one of the above factors vanishes, then  $\nu_3=0$ which it is not possible.   
\item If $a_1 e_{11}+e_{33}=0$ and $b_1\not=0$, then $A_3=-e_{11}^2(1+a_1^2+b_1^2)$ and $A_3=0$ gives a contradiction.
\item Suppose  the parenthesis of $B_2$  vanishes and $b_1\not=0$ and $a_1e_{11}+e_{33}\not=0$. If $a_1\not=0$, then we get $e_{33}=(a_1^2+b_1^2+(1-a_1^2)e_{11}^2)/(2 a_1 e_{11})$. With this value of $e_{33}$, identity $e_{11}^2+e_{33}^2-1=0$ gives a contradiction. Thus $a_1=0$. Hence $B_2=2b_1e_{11}e_{33}(1+b_1^2-e_{33}^2)$. Since $e_{33}^2\leq 1$ and $b_1\not=0$, we conclude $e_{33}=0$. In particular, $e_{11}^2= 1$. But now $A_2=-(1+b_1^2)^2$, obtaining a contradiction.
\end{itemize}
\item Case $r(s)=c_2\sqrt{2s+c_2}$, $c_1>0$.. After a translation on the parameter $s$, which it only  changes the value $a_0$ of the function $a(s)=a_1 s+a_0$  to another one, we can assume $c_2=0$. Thus, let $r(s)=c_1\sqrt{s}$ for some $c_1>0$. Then $B_3=0$ simplifies into
\begin{equation}\label{s1}
b_1  e_{11} \left(e_{11}  (1+   b_1^2+2 e_{33}^2 - a_1^2  (2 e_{33}^2+1 ))+2 a_1 e_{11}^2 e_{33}-2 a_1 e_{33} \left(e_{33}^2+1\right)\right)=0.
\end{equation}
\begin{itemize}
\item Case $b_1=0$. Then $A_3=0$ writes 
\begin{equation}\label{s22}
 (a_1e_{33}-e_{11})(e_{11}e_{33}a_1^2+(1-e_{11}^2+e_{33}^2)a_1-e_{11}e_{33})=0.
 \end{equation}
If $ a_1 e_{33}-e_{11}=0$, then $A_2=0$ becomes $a_1^2(a_1^2+2)e_{33}=0$. If $a_1=0$, then \eqref{h2} implies $e_{33}^2(1+e_{33}^2)=0$ hence $\nu_3=0$. If $a_1\not=0$ and $e_{33}=0$ then $\nu_3=0$ again. Definitively,  the second parenthesis in \eqref{s22}vanishes identically. 
A first observation is that it is immediate that $e_{33}=0$ it is not possible. Thus if the parenthesis is $0$, then $e_{11}(a_1^2-1)+2a_1 e_{33}=0$.
 In particular, $a_1^2\not=0,1$. Then $e_{33}=e_{11}(1-a_1^2)/(2a_1)$ and the relation $e_{11}^2+e_{33}^2=1$ yields $e_{11}^2=4 a_1^2/(1+a_1^2)^2$. Now $A_2=0$ implies $a_1^2c_2^2(a_1^2-1)=0$, obtaining a contradiction.
 \item Case $b_1\not=0$.  The coefficient $B_1$ writes 
 $$B_1=\frac{b_1 c_1}{2}\left( \frac{b_{11}}{s^{1/2}}+\frac{c_1^2e_{11}^2(2+e_{33}^2)}{s^{3/2}}\right),$$
 for some constants  $b_{11}$.  Thus $B_1=0$ implies $c_1^2e_{11}^2(2+e_{33}^2)=0$ and this gives a contradiction.  
 \end{itemize}
 \end{enumerate}
\item Case $rr''+r'^2\not=0$. Then $A_4=B_4=0$ imply
\begin{equation*}
\begin{split}
(a_1^2-b_1^2-1)e_{11}+2a_1 e_{33}&=0,\\
 b_1 (a_{1}e_{11}+e_{33})&=0.
 \end{split}
 \end{equation*}
If $b_1\not=0$, the solution of the system is $e_{11}=e_{33}=0$, which it is not possible. Thus $b_1=0$ and $(a_1^2-1)e_{11}+2a_1 e_{33}=0$. The case $a_1=0$ is discarded because $A_4=0$ gives $e_{11} (rr''+r'^2)=0$. Therefore we get $e_{33}=(1-a_1^2)e_{11}/(2a_1)$. Then $A_4=0$ gives $e_{11}^2=4a_1^2/(1+a_1^2)$. Now $A_3=0$ is $a_1^3r'(r'^2+rr'')=0$ which it is not possible.
\end{enumerate}
\item Subcase $b''\not=0$. Then $A_5=0$ implies $e_{33}=-a_1 e_{11}$. Moreover, $e_{11}^2+e_{33}^2=1$ gives $e_{11}^2=1/(1+a_1^2)$.  Now the coefficient $B_5$ is 
$$B_5=\frac{rb''}{8(1+a_1^2)^2}(1+a_1^2+b'^2)$$
and $B_5=0$ gives a contradiction.
\end{enumerate}
\item Case $a''\not=0$. Then one of the two parentheses of \eqref{pa} vanishes identically. We will assume that it is the first one and a similar argument applies if it is   the second one. Since $e_{33}\not=1$ because $e_{11}\not=0$, then $b'=0$ identically and $a'=\frac{e_{11}}{e_{33}-1}$. Since $e_{11}$ and $e_{33}$ are constant because $\mathcal{B}$ is constant, we obtain  $a''=0$, a contradiction. 
 \end{enumerate}
  
 \begin{remark} As we have observed, the general case $\Lambda\not=0$ is very more difficult. We prove in this Remark that it is not possible that  the planes of the foliation are orthogonal to the vector $e_3$. In such a case, and using the above notation, we now have $e_3=(1,0,0)$ in coordinates with respect to $\mathcal{B}$, that is, $e_{11}=1$ and $e_{33}=0$.   Then
  \begin{equation*}
  \begin{split}
  A_5&=\frac{1}{16} r \left(-4 a' b' b''-\left((a'^2-b'^2-1) (\Lambda -2 a'')\right)\right),\\
  B_5&=\frac{1}{8} r \left(b'' (a'^2-b'^2-1)-a' b' (\Lambda -2 a'')\right).
  \end{split}
  \end{equation*}
  Viewing $A_5=0$ and $B_5=0$ as a linear system on $a''$ and $b''$, the determinant of the coefficients is
  $$\frac{1}{64}r^2\left(b'^4+2b'^2(a'^2+1)+(a'^2-1)^2\right).$$
  Then this determinant is positive or, otherwise, $b'^2=0$ and $a'^2=1$. In any of the two cases, we conclude $a''=b''=0$. Thus, 
  $a(s)=a_1 s+a_0$ and $b(s)=b_1 s+b_0$, for constants $a_i,b_i$, $i=1,2$. Now $A_4=-\frac14\Lambda rr'$. Then $A_4=0$  implies that $r$ is a constant function. Then $A_3=-\frac14 r\Lambda$ and $A_3=0$ gives the desirable contradiction.
  \end{remark}

  \section{Proof of Thm. \ref{ROT1}}\label{s5}

Suppose that $\Sigma$ is a surface of revolution. In particular, the planes of the foliation are parallel. Thus we can  follow the same arguments and notation as in Subsect. \ref{32}. The proof is by contradiction and suppose, without loss of generality, that  $e_{11}\not=0$.  Since the surface is rotational, we can assume that $a(s)=b(s)=0$. Recall that $\Lambda$ is any real constant.   Equation \eqref{h2} can be expressed  as 
\begin{equation}\label{pol}
\sum_{n=0}^5 A_n(s)\cos(n\theta)=0.
\end{equation}
For $n=5$, we have $A_5=e_{11}^5\Lambda r/16$. Thus if $\Lambda\not=0$, we obtain $e_{11}=0$ proving the result. 

In the case$\Lambda=0$, we are under the situation of Thm. \ref{T2} where it was proved that the rotation axis is parallel to $e_3$.  

\begin{remark} Notice that the brevity of the proof of Thm. \ref{ROT1} hides the discussion of the case $\Lambda=0$ in Thm. \ref{T2}. As we said in the Introduction, for general axially symmetric anisotropic energies $\mathcal{F}=\mathcal{F}(\nu_3)$, it is not known if the axis of a CAMC surface of revolution is parallel to the $z$-axis.  
\end{remark}

\section*{Acknowledgements}
The author   has been partially supported by Grant PID2023-150727NB-I00 funded by MICIU/AEI/10.13039/501100011033, and ERDF/EU and 
Grant PID2023-150727NB-I00 and Maria de Maeztu Unit of Excellence IMAG, reference CEX2020-
001105-M, funded by MICIU/AEI/10.13039/501100011033, and ERDF/EU. 
 
\section*{Statements and Declarations}  

{\bf Conflict of interest.} The author declares that they have no conflict of interest.

{\bf Data Availability.} Data sharing not applicable to this article as no datasets were generated or analysed during the current study.


\begin{thebibliography}{99}


\bibitem{bs} E. Barbosa and L. C.  Silva, Surfaces of constant anisotropic mean curvature with free boundary in revolution surfaces,  \emph{Manuscr. Math.} \textbf{169} (2022), 439--459.

\bibitem{en1} A. Enneper, \"{U}eber die cyclischen Fl\"{a}chen,   Nach. K\"{o}nigl. Ges. d. Wissensch. G\"{o}ttingen, \emph{Math. Phys. Kl.} (1866) 243--249. 

\bibitem{en2} A. Enneper, Die cyklischen Fl\"{a}chen,   \emph{Z. Math. Phys.} \textbf{14} (1869), 393--421. 

 
 

\bibitem{gmt} J. A. G\' alvez, P. Mira and    M. P. Tassi, Complete surfaces of constant anisotropic mean curvature,  \emph{Adv. Math.} \textbf{428} (2023), Paper No. 109137. 

\bibitem{gx} J. Guo and  C. Xia, Stable anisotropic capillary hypersurfaces in a half-space,   arXiv:2301.03020 [math.DG]  
 


\bibitem{ja1} W. Jagy, Minimal hypersurfaces foliated by spheres,  \emph{Michigan Math. J.} \textbf{38} (1991), 255--270.   
\bibitem{ja2} W. Jagy, Sphere-foliated constant mean curvature submanifolds,  \emph{Rocky Mountain J. Math.} \textbf{28} (1998), 983--1015.

\bibitem{jw} X. Jia, G. Wang and  C. Xia,   X. Zhang, Alexandrov's theorem for anisotropic capillary hypersurfaces in the half-space,  \emph{Arch. Ration. Mech. Anal.} \textbf{ 247} (2023),  25.

\bibitem{kp1} M.  Koiso and  B.  Palmer, Geometry and stability of surfaces with constant anisotropic mean curvature,  \emph{Indiana Univ. Math. J.} \textbf{54} (2005),   1817--1852.


\bibitem{kp2} M.  Koiso and  B.  Palmer, Stability of anisotropic capillary surfaces between two parallel planes,  \emph{Calc. Var. Partial Differ. Equ.} \textbf{25},  (2006), 275--298.


\bibitem{kp3} M.  Koiso and B.  Palmer, Uniqueness theorems for stable anisotropic capillary surfaces,  \emph{SIAM J. Math. Anal.} \textbf{39} (2007), 721--741.

\bibitem{kp4} M.  Koiso and  B.  Palmer, Rolling construction for anisotropic Delaunay surfaces,  \emph{Pacific J. Math.} \textbf{234} (2008),   345--378.

\bibitem{kp5} M.  Koiso and  B.  Palmer, Equilibria for anisotropic surface energies with wetting and line tension,  \emph{Calc. Var. Partial Differ. Equ.}\textbf{43} (2012), 555--587.



\bibitem{lls} F. J. L\'opez, R. L\'opez and    R. Souam, Maximal surfaces of Riemann type in Lorentz-Minkowski space  $L^3$,  \emph{Michigan Math. J.} \textbf{47} (2000),  469--497.


\bibitem{lo2} R. L\'opez,  Constant mean curvature hypersurfaces foliated by spheres,  \emph{Differential Geom. Appl.} \textbf{11} (1999), 245--256. 

\bibitem{lo3} R. L\'opez,  Constant mean curvature surfaces foliated by circles in Lorentz-Minkowski space,  \emph{Geom. Dedicata} \textbf{76} (1999), 81--95.

\bibitem{lo1} R. L\'opez,  Cyclic surfaces of constant Gauss curvature,  \emph{Houston J. Math.} \textbf{27} (2001), 799--805.

\bibitem{lo4} R. L\'opez,  Cyclic hypersurfaces of constant curvature, \emph{Advanced Studies in Pure Mathematics}, \textbf{34}  2002, Minimal Surfaces, Geometric Analysis and Symplectic Geometry, 185--199.


\bibitem{ni} J.C.C. Nitsche, Cyclic surfaces of constant mean curvature,  \emph{Nachr. Akad. Wiss. Gottingen Math. Phys. II} \textbf{1} (1989) 1--5. 
\bibitem{pa} S.-H. Park,  Sphere-foliated minimal and constant mean curvature hypersurfaces in space forms and Lorentz-Minkowski space. \emph{Rocky Mountain J. Math.} \textbf{32} (2002), 1019? 1044.

\bibitem{re} R. C.  Reilly,   The relative differential geometry of nonparametric hypersurfaces, \emph{Duke Math.J.} \textbf{43} (1976),  705--721.


\bibitem{ri} B. Riemann,  \"{U}ber die Fl\"{a}chen vom kleinsten Inhalt bei gegebener Begrenzung,   \emph{Abh. K\"{o}nigl. Ges. d. Wissensch. G\"{o}ttingen, Mathema.} \textbf{13} (1868), 329--333. 

\bibitem{ro} C. Rosales, Compact anisotropic stable hypersurfaces with free boundary in convex solid cones, 
\emph{Calc. Var. Partial Differ. Equ.} \textbf{62} (2023),  Paper No. 185, 20 pp.

\bibitem{ta} J. E. Taylor, Crystalline variational problems,  \emph{Bull. Amer. Math. Soc.} \textbf{84} (1978), 568--588. 

\bibitem{wo} Wolfram Research, Inc. Mathematica,    Version 13.3.  Champaign, IL (2023).


\end{thebibliography}
\end{document}